\documentclass[letterpaper, 10 pt, conference]{ieeeconf}  % Comment this line out if you need a4paper

\IEEEoverridecommandlockouts                              % This command is only needed if 
\usepackage{amsmath,amssymb,bm,bbm,mathrsfs,amscd}
\usepackage{color}
\usepackage{dsfont}
\usepackage{graphicx}
\usepackage{epsfig}
\usepackage{subfigure}
\usepackage{algorithm}
\usepackage[noend]{algpseudocode}
\usepackage{tikz}
\usepackage{mathtools}

\newtheorem{theorem}{Theorem}
\newtheorem{definition}{Definition}
\newtheorem{proposition}{Proposition}
\newtheorem{lemma}{Lemma}
\newtheorem{corollary}{Corollary}
\newtheorem{assumption}{Assumption}
\newtheorem{standassumption}{Standing Assumption}

\newcommand{\RR}{{\mathbb{R}}}

\newcommand{\mc}{\mathcal}

\newcommand{\op}{\operatorname}

\newcommand{\prox}{\operatorname{prox}}

\newcommand{\bs}{\boldsymbol}
\newcommand{\fineass}{\hfill\small$\blacksquare$}

\newcommand{\scrX}{\mathcal{X}}

\newcommand{\setH}{\mathsf{H}}

\DeclarePairedDelimiter{\abs}{\lvert}{\rvert}
\DeclarePairedDelimiter{\inner}{\langle}{\rangle}
\DeclarePairedDelimiter{\norm}{\lVert}{\rVert}

\DeclareMathOperator{\Zer}{zer}
\DeclareMathOperator{\Id}{Id}

\DeclareMathOperator{\NC}{\mathsf{N}}
\DeclareMathOperator{\diag}{diag}
\DeclareMathOperator{\Span}{span}
\DeclareMathOperator{\gr}{gph} % Rockafellar notation

\DeclareMathOperator{\dom}{dom}

\newcommand{\R}{\mathbb{R}}
\newcommand{\0}{\mathbf{0}}
\newcommand{\1}{\mathbf{1}}
\newcommand{\bW}{\mathbf{W}}
\newcommand{\bL}{\mathbf{L}}
\newcommand{\bA}{{\mathbf A}}

\title{\LARGE \bf
Distributed forward-backward (half) forward algorithms for generalized Nash equilibrium seeking
}

\author{Barbara Franci$^{1}$, Mathias Staudigl$^{2}$ and Sergio Grammatico$^{1}$% <-this % stops a space
\thanks{$^{1}$Barbara Franci and Sergio Grammatico are with the Delft Center for System and Control, TU Delft, The Netherlands
        {\tt\small \{b.franci-1, s.grammatico\}@tudelft.nl}.}%
\thanks{$^{2}$Mathias Staudigl is with the Department of Data Science and Knowledge Engineering, Maastricht University, The Netherlands
{\tt\small m.staudigl@maastrichtuniversity.nl}.}%
\thanks{This work was partially supported by NWO under research projects OMEGA (613.001.702) and P2P-TALES (647.003.003), and by the ERC under research project COSMOS (802348). M. Staudigl acknowledges support from the COST Action CA16228 "European Network form Game Theory".}}

\begin{document}

\maketitle
\thispagestyle{empty}
\pagestyle{empty}

\begin{abstract}
%We consider the development of schemes for distributed computation of generalized Nash equilibria of noncooperative games over networks. We consider games in which the strategy spaces of all players are coupled together by globally shared affine constraints, and we use the associated variational inequality as an equilibrium selection device.
%\textcolor{red}{MS: Needs more details}
%Computing generalized Nash equilibria in a distributed networked system with competing agents is a challenging task with important applications in engineering, economics and operations research. 
We present two distributed algorithms for the computation of a generalized Nash equilibrium in monotone games. The first algorithm follows from a forward-backward-forward operator splitting, while the second, which requires the pseudo-gradient mapping of the game to be cocoercive, follows from the forward-backward-half-forward operator splitting. Finally, we compare them with the distributed, preconditioned, forward-backward algorithm via numerical experiments. 
\end{abstract}

\section{Introduction}

Generalized Nash equilibrium problems (GNEPs) have been widely studied in the literature \cite{Ros65, facchinei2007, facchinei2010} and such a strong interest is motivated by numerous applications ranging from economics to engineering \cite{pavel2007,kulkarni2012}. In a GNEP, each agent seeks to minimize his own cost function, under local and coupled feasibility constraints. In fact, both the cost function and the constraints depend on the strategies chosen by the other agents. Due to the presence of these shared constraints, the search for generalized Nash equilibria is usually a quite challenging task. 

For the computation of a GNE, various algorithms have been proposed, both distributed \cite{belgioioso2018,yi2019}, and semi-decentralized \cite{facchinei2010,belgioioso2017}. When dealing with coupling constraints, a common principle is the focus on a special class of equilibria, which reflect some notion of fairness among the agents. This class is known as \emph{variational equilibria} (v-GNE) \cite{Deb52,facchinei2010}. Besides fairness considerations, v-GNE is computationally attractive since it can be formulated in terms of variational inequality, which makes it possible to solve them via operator splitting techniques \cite{BauCom16,facchinei2007}. A recent breakthrough along these lines is the distributed, preconditioned, forward-backward (FB) algorithm conceived in \cite{yi2019} for strongly-monotone games. The key lesson from \cite{yi2019} is that the FB method cannot be directly applied to GNEPs, thus a suitable preconditioning is necessary. From a technical perspective, the FB operator splitting requires that the pseudo-gradient mapping of the game is strongly monotone, an assumption which is not always satisfied.% even in simple matrix games without joint constraints. %which is not always the case in monotone games, even with linear cost functions \cite{grammatico18}. \MSTnote{Actually, cocoercivity can never hold for zero sum matrix games without coupling constraint. The map corresponding to this game is skew.}

In this paper we investigate two distributed algorithmic schemes for computing a v-GNE. Motivated by the need to relax the strong monotonicity assumption on the pseudo-gradient of the game, we first investigate a distributed \textit{forward-backward-forward} (FBF) algorithm \cite{tseng2000}. We show that a suitably constructed FBF operator splitting guarantees not only fully distributed computation, but also convergence to a v-GNE under the mere assumption of monotonicity of the involved operators. This enables us to drop the strong monotonicity assumption, which is the main advantage with respect to the FB-based splitting methods \cite{yi2019}. %Another technical drawback of the FB splitting proposed in \cite{yi2019} is that the algorithm is not a full splitting scheme. 
As a second condition, in order to exploit the structure of the monotone inclusion defining the v-GNE problem, we also investigate the \textit{forward-backward-half-forward} (FBHF) algorithm \cite{briceno2018}. We would like to point out that both our algorithms are distributed in the sense that each agent needs to know only his local cost function and its local feasible set, and there is no central coordinator that updates and broadcasts the dual variables. The latter is the main difference with semi-decentralized schemes for aggregative games \cite{Gram17,belgioioso2017}.

% , we propose two distributed algorithms based on operator splitting for computing a v-GNE in monotone games. Specifically, without the additional assumption of strong-monotone pseudo-gradient mapping, we propose a distributed \textit{forward-backward-forward} (FBF) algorithm \cite{tseng2000}. Second  pseudo-gradient mapping, we present a distributed \textit{forward-backward-half-forward} (FBHF) algorithm \cite{briceno2018}. Both our algorithms are fully distributed in sense that each agent needs to know only his local cost function and its local feasible set, and there is no central coordinator that updates and broadcasts the dual variables. The latter is the main difference with semi-decentralized schemes for aggregative games \cite{Gram17,belgioioso2017}. Moreover, our algorithms do not need a preconditioning procedure. Our main technical results are thus to show global convergence of these two algorithms to a variational GNE of the (strongly) monotone game, for suitable step sizes.

Compared with the FB and the FBHF algorithms, the FBF requires less restrictive assumptions to guarantee convergence, i.e., plain monotonicity of the pseudo-gradient mapping. On the other hand, the FBF algorithm requires two evaluations of the pseudo-gradient mapping, which means that the agents must communicate at least twice at each iterative step. Confronted with the FBF algorithm, our second proposal, the FBHF algorithm requires only one evaluation of the pseudo-gradient mapping, but needs strong monotonicity to provide theoretical convergence guarantees. Effectively, the FBHF algorithm is guaranteed to converge under the same assumptions as the preconditioned FB \cite{yi2019}. %From a computational perspective, the FBHF and the FB should perform similarly. In our numerical simulations, the FBHF algorithm shows faster convergence. Indeed, the convergence analysis shows that it can tolerate slightly larger step sizes.

\paragraph*{Notation}
%We follow the standard notation of variational analysis \cite{BauCom16}. 
$\RR$ indicates the set of real numbers and $\bar\RR=\RR\cup\{+\infty\}$. $\0_N$ ($\1_N$) is the vector of N zeros (ones). The Euclidean inner product and norm are indicated with $\inner{\cdot,\cdot}$ and $\norm{\cdot}$, respectively.
Let $\Phi$ be a symmetric, positive definite matrix, $\Phi \succ 0$. The induced inner product is $\inner{\cdot,\cdot}_{\Phi}:=\inner{\Phi\cdot,\cdot}$, and the associated norm is $\norm{\cdot}_{\Phi}:=\inner{\cdot,\cdot}_{\Phi}^{1/2}$. We call $\mc H_\Phi$ the Hilbert space with norm $\norm{\cdot}_\Phi$. 
Given a set $\mc X\subseteq\RR^n$, the normal cone mapping is denoted with $\NC_{\scrX}(x)$. $\Id$ is the identity mapping. Given a set-valued operator $A$, the graph of $A$ is the set $\gr(A)=\{(x,y)\vert y\in Ax\}$ The set of zeros is $\Zer A=\{x\in\RR^n \mid 0\in Ax\}$.  
The resolvent of a maximally monotone operator $A$ is the map $\mathrm{J}_{A}=(\Id+A)^{-1}$. If $g:\RR^{n}\to(-\infty,\infty]$ is a proper, lower semi-continuous, convex function, its subdifferential is the maximal monotone operator $\partial g(x)$. The proximal operator is defined as $\prox^{\Phi}_{g}(v)=\op{J}_{\Phi\partial g}(v)$ \cite{BauCom16}.
Given $x_{1}, \ldots, x_{N} \in \RR^{n}, \boldsymbol{x} :=\op{col}\left(x_{1}, \dots, x_{N}\right)=\left[x_{1}^{\top}, \dots, x_{N}^{\top}\right]^{\top}$.

\section{Mathematical Setup: The Monotone Game and Variational Generalized Nash Equilibria}
\label{sec:problem}
We consider a game with $N$ agents where each agent chooses an action $x_{i}\in\RR^{n_i}$, $i\in\mc I=\{1,\dots,N\}$.  

%from its local decision set $\Omega_{i} \subseteq \R^{n_{i}}$. Let us define the product space $\bs\Omega :=\Omega_{1}\times\ldots\times\Omega_{N}$ and $n :=\sum_{i=1}^N n_i$. \MSTnote{This is inconsistent. If $g_{i}$ embodies constraints why do we need $\Omega_{i}$ here?}
% with inner product 
%\vspace{-.15cm}\begin{equation}
%(x,y)\mapsto \inner{x,y}:=\sum_{i=1}^{N}\inner{x_{i},y_{i}}_{i}.
%\vspace{-.15cm}\end{equation}
Each agent $i$ has an extended-valued local cost function $J_{i}: \RR^n \to (-\infty,\infty]$ of the form 
\vspace{-.15cm}\begin{equation}\label{eq:f}
J_{i}(x_{i}, \bs x_{-i}):=f_{i}(x_{i},\bs x_{-i}) + g_{i}(x_{i}).
\vspace{-.15cm}\end{equation}
where $\bs x_{-i}=\op{col}(\{x_j\}_{j\neq i})$ is the vector of all decision variables except for $x_i$, and $g_{i}:\RR^{n_{i}}\to(-\infty,\infty]$ is a local idiosyncratic costs function which is possibly non-smooth. Thus, the function $J_{i}$ in \eqref{eq:f} has the typical splitting into smooth and non-smooth parts. %We assume that the non-smooth part is captured by the function $g_{i}:\R^{n_i} \to \bar{\RR}$, which can model not only a local cost, but also local constraints.
\begin{standassumption}[Local cost]
For each $i\in\mc I$, the function $g_i$ in \eqref{eq:f} is lower semicontinuous and convex.
For each $i\in\mc I$, $\dom(g_{i})=\Omega_i$ is a closed convex set.
\fineass
%$g_{i}\in\Gamma_{0}(\scrH_{i})$ for all $i\in\{1,\ldots,N\}$.\tc{blue}{sergio}\hfill$\blacksquare$
\end{standassumption}

%By convexity of $g_{i}$, it follows that $\dom(g_{i})$ is convex as well. 
%A classical example for the local cost function is the indicator function of the local feasible set, i.e., $g_{i}(x_{i})=0$ if $x_{i}\in \Omega_{i}$, and $g_{i}(x_{i}) = \infty$ otherwise. Other examples are penalty functions to promote sparsity, or other desirable structure. For the function $f_i$ in \eqref{eq:f}, we assume convexity and differentiability, as usual in the GNEP literature \cite{facchinei2010}.

Examples for the local cost function are indicator functions to enforce set constraints, or penalty functions that promote sparsity, or other desirable structure. 

For the function $f_i$ in \eqref{eq:f}, we assume convexity and differentiability, as usual in the GNEP literature \cite{facchinei2010}.
\begin{standassumption}[Local convexity]
\label{ass:IC}
For each $i \in \mathcal{I}$ and for all $\bs{x}_{-i} \in \RR^{n-n_i}$, the function $f_{i}(\cdot, \bs{x}_{-i})$ in \eqref{eq:f} is convex and continuously differentiable. 
\hfill$\blacksquare$
\end{standassumption}

%\begin{assumption}[Global Monotonicity]
%\label{ass:GM}
%The functions $h_{1},\ldots,h_{N}$ satisfy the joint monotonicity property
%\vspace{-.15cm}\begin{equation}\label{eq:h-monotone}
%\sum_{i=1}^{N}\inner{\nabla_{i}F(\boldsymbol x)-\nabla_{i}h(x'),x_{i}-x'_{i}}_{i}\geq 0
%\vspace{-.15cm}\end{equation}
%for all $x,x'\in\Omega$.
%\end{assumption}

%Given $\mc S=\prod_{i=1}^N\mc S_i$, as coupling constraints, 
We assume that the game displays joint convexity with affine coupling constraints defining the collective feasible set 
\vspace{-.15cm}\begin{equation}\label{eq:coupling}
\bs{\mc{X}}:=\left\{\bs x \in\bs\Omega \mid A \bs{x}-b \leq {\boldsymbol{0}}_{m}\right\}
\vspace{-.15cm}\end{equation}
where $A:=[A_1,\dots, A_N]\in\RR^{m\times n}$ and $b:=\sum_{i=1}^{N}b_{i}\in\RR^m$. 
Effectively, each matrix $A_i\in\RR^{m\times n_i}$ defines how agent $i$ is involved in the coupling constraints, thus we consider it to be private information of agent $i$. 
% agent specific linear mappings $A_{i}:\Omega_{i}\to\R^{m}$, and the polyhedral set 
%\vspace{-.15cm}\begin{equation}\label{eq:coupling}
%\scrF=\left\{x\in\Omega: \sum_{i=1}^{N}A_{i}x_{i}\leq \sum_{i=1}^{N}b_{i}\right\}. 
%\vspace{-.15cm}\end{equation}
%The matrix $A_i$ defines how agent $i$ is involved in the coupling constraints.
%Define the collective feasible set by
%\vspace{-.15cm}\begin{equation}
%\scrX:=\Omega\cap\scrF. 
%\vspace{-.15cm}\end{equation}
Then, for each $i$, given the strategies of all other agents $\bs x_{-i}$, the feasible decision set is
\vspace{-.15cm}\begin{equation}
\scrX_{i}(\bs{x}_{-i}) := \left\{y_i \in \Omega_i \mid \, A_i y_i \leq b-\textstyle\sum_{j \neq i}^{N} A_j x_j\right\}.
\vspace{-.15cm}\end{equation}

Next, we assume a constraint qualification condition.

\begin{standassumption}\label{ass_X}
(\textit{Constraint qualification})
%For each $i \in \mc{N},$ the set $\Omega_{i}$ is nonempty, compact and convex.
The set $\bs{\mc{X}}$ in \eqref{eq:coupling} satisfies Slater's constraint qualification. 
\fineass
\end{standassumption}
The aim of each agent is to solve its local optimization problem
\vspace{-.15cm}\begin{equation}\label{game}
\forall i\in\mc I: \quad\left\{\begin{array}{cl}
\min_{x_i \in  \Omega_i} & J_{i}(x_i, \bs x_{-i}) \\ 
\text { s.t. } & A_i x_i \leq b-\sum_{j \neq i}^{N} A_j x_j.
\end{array}\right.
%\begin{cases}
%\min\limits_{x_{i}\in\mc \Omega_i}J_{i}(x_i, \bs x_{-i})\\
%\text{ s.t. } x_{i}\in \mc X_{i}(\boldsymbol x_{-i}). 
%\end{cases}
\vspace{-.15cm}\end{equation}

Thus, the solution concept for such a competitive scenario is the generalized Nash equilibrium \cite{Deb52,facchinei2010}. 

\begin{definition} (\textit{Generalized Nash equilibrium})
A collective strategy $\bs x^{\ast}=\op{col}(x_{1}^{\ast},\ldots,x_{N}^{\ast})\in \bs{\mc{X}}$
is a generalized Nash equilibrium of the game in \eqref{game} if, for all $i\in\mc I$,
\vspace{-.15cm}\begin{equation*}
J_i(x^*_i,\boldsymbol x^*_{-i}) \leq \inf\{ J_i(y,\boldsymbol x^*_{-i}) \, \mid  \, y\in\mc X_i(\bs x_{-i})\}.\vspace{-.4cm}
\vspace{-.15cm}\end{equation*}
\fineass
\end{definition}
%In other words, a GNE is a set of decision variables where no agent can decrease its cost by unilaterally deviating from its strategy. 
To derive optimality conditions characterizing GNE, we define agent $i$'s Lagrangian function as
%\vspace{-.15cm}\begin{equation}
$\mc L_{i}(x_{i},\lambda_i,  \boldsymbol x_{-i}):=J_{i}(x_i, \bs x_{-i})+\lambda_i^{\top}(A\boldsymbol{x}-b)$
%\vspace{-.15cm}\end{equation}
where $\lambda_i\in\RR^m_{\geq 0}$ is the Lagrange multiplier associated with the coupling constraint $A \bs{x} \leq b$. Thanks to the sum rule of the subgradient for Lipschitz continuous functions \cite[\S 1.8]{Cla98}, we can write the subgradient of agent $i$ as 
$ \partial_{x_{i}}J_{i}(x_{i},\boldsymbol x_{-i})=\nabla_{x_{i}} f_{i}(x_{i},\boldsymbol x_{-i})+\partial g_{i}(x_{i})$. Therefore, 
Under Assumption \ref{ass_X}, the Karush--Kuhn--Tucker (KKT) theorem ensures the existence of a pair $(x^*_{i},\lambda^*_i)\in\Omega_{i}\times\R^{m}_{\geq 0}$, such that 
%\vspace{-.15cm}\begin{equation}\label{KKT_no_inclusion}
%\forall i\in\mc I: \quad\begin{cases}
%\0_{n_i} \in\partial_{i}J_{i}(x^*_{i},\bs x_{-i})+A^{\top}_{i}\lambda^*_i\\
%A\boldsymbol{x}^*-b\leq \0_{m}.
%\end{cases}
%\vspace{-.15cm}\end{equation}
%We recall that the stationarity conditions and the complementary slackness conditions can be efficiently written as a parallel inclusion. Thanks to the sum rule of the subgradient for Lipschitz continuous functions \cite[\S 1.8]{Cla98}, we can write the subgradient of agent $i$ as 
%$ \partial_{i}J_{i}(x_{i},\boldsymbol x_{-i})=\nabla_{i} f_{i}(x_{i},\boldsymbol x_{-i})+\partial g_{i}(x_{i})$. Then, %\eqref{KKT_no_inclusion} can be equivalently written as 
\begin{equation}\label{KKT_game}
\forall i\in\mc I:\begin{cases}
\0_{n_i}\in \nabla_{x_{i}} f_{i}(x^*_{i};\boldsymbol x^*_{-i})+\partial g_{i}(x^*_{i})+A^{\top}_{i}\lambda_i^*\\
\0_{m}\in \NC_{\R^{m}_{\geq0}}(\lambda_i^*)-(A\boldsymbol{x}^*-b).
\end{cases}
\vspace{-.15cm}\end{equation}

We conclude the section by postulating a standard assumption for GNEP's \cite{facchinei2010}, and inclusion problems in general \cite{BauCom16}, concerning the monotonicity and Lipschitz continuity of the mapping that collects the partial gradients $\nabla_{i} f_{i}$.

\begin{standassumption}[Monotonicity]
\label{ass:GM}
The mapping 
\vspace{-.15cm}\begin{equation}\label{eq:F}
F(\boldsymbol x):= \mathrm{col}\left( \nabla_{x_1}f_{1}(\bs x),\ldots,\nabla_{x_N}f_{N}(\bs x)\right)
\vspace{-.15cm}\end{equation}
is monotone on $\bs\Omega$, i.e., for all $\bs x,\bs y\in\bs\Omega$,%\vspace{-.15cm}
%\vspace{-.15cm}\begin{equation*}\label{eq:h-monotone}
$\langle F(\bs x)-F(\bs y),\bs x-\bs y\rangle\geq 0.$
%\vspace{-.15cm}\end{equation*} 
and $\frac{1}{\beta}$-Lipschitz continuous, $\beta > 0$, i.e., for all $\bs x,\bs y\in\bs\Omega$,\vspace{-.15cm}
%\vspace{-.15cm}\begin{equation*}
$\norm{F(\bs x)-F(\bs y)} \leq \tfrac{1}{\beta}\norm{\bs x-\bs y}.$%\vspace{-.3cm}
%\vspace{-.15cm}\end{equation*}
\fineass
\end{standassumption}
%%%%%%%%
\vspace{.15cm}
Among all possible GNEs of the game, this work focuses on the computation of a \emph{variational GNE} (v-GNE) \cite[Def. 3.10]{facchinei2010}, i.e. a GNE in which all players share consensus on the dual variables: %, primal strategies that solve the KKT systems in \eqref{KKT_game} with the same Lagrange multiplier \cite[Th. 3.1]{facchinei2007vi}, \cite[Th. 3.1]{auslender2000}:
\vspace{-.15cm}\begin{equation}\label{KKT_VI}
\forall i\in\mc I:\begin{cases}
\0_{n_i}\in \nabla_{x_{i}}f_{i}(x^{*}_{i};\boldsymbol{x}^{*}_{-i})+\partial g_{i}(x^{*}_{i})+A_{i}^{\top}\lambda^{\ast}\\
\0_{m}\in \NC_{\R^{m}_{\geq0}}(\lambda^*)-(A\boldsymbol{x}^*-b).\\
\end{cases}
\vspace{-.15cm}\end{equation}

\section{Distributed Generalized Nash equilibrium seeking via Operator Splitting}
In this section, we present the proposed distributed algorithms. We allow each agent to have information on his own local problem data only, i.e., $J_{i},\Omega_{i}, A_{i}$ and $b_{i}$. We let each agent $i$ control its local decision $x_{i}$, and a local copy $\lambda_{i}\in\R^{m}_{\geq0}$ of dual variables, as well as a local auxiliary variable $z_{i}\in\R^{m}$ used to enforce consensus of the dual variables. %Since each cost function depends on the decision variables of other agents, we indicate with $\mc N_i^J$ the set of agents $j\in\mc I$ such that $J_i$ depends explicitly on $x_j$.
To actually reach consensus on the dual variables, we let the agents exchange information via an undirected weighted \emph{communication graph}, represented by its weighted adjacency matrix $\boldsymbol W = [w_{i,j}]_{i,j}\in\RR^{N\times N}$. We assume $w_{ij}>0$ iff $(i,j)$ is an edge in the communication graph. The set of neighbours of agent $i$ in the graph is $\mc{N}_{i}^{\lambda}=\{j |w_{i,j}>0\}$.
\vspace{.15cm}
\begin{standassumption}[Graph connectivity]\label{ass:graph}
The matrix $\bW$ is symmetric and irreducible.\fineass
%satisfies the following conditions: 
%\begin{enumerate}
%\item There are no self-referential loops, meaning that $w_{ii}=0$ for all $i\in\mc I$;
%\item Weights are symmetric in the sense that $w_{i,j}=w_{ji}$ for all $i,j\in\mc I$.
%\item The graph is connected.%, meaning that the matrix $\bW$ is primitive:  there exists $n\geq 1$ such hat $\bW^{n}$ is a matrix consisting of only positive entries. 
%\end{enumerate}
\end{standassumption}
Define the weighted Laplacian as $\bL:=\diag\left\{(\bW\1_{N})_{1}, \dots, (\bW\1_{N})_{N}\right\}-\bW$. It holds that $\bL^{\top}=\bL$, $\ker(\bL)=\Span(\1_{N})$ and that, given Standing Assumption \ref{ass:graph}, $\bL$ is positive semi-definite with real and distinct eigenvalues $0=s_{1}<s_{2}\leq \ldots \leq s_{N}$.
Moreover, given the maximum degree of the graph $\mc G^\lambda$, $\Delta:=\max_{i\in\mc I}(\bW\1_{N})_{i}$, it holds that $\Delta \leq s_{N}\leq 2\Delta$. Denoting by $\kappa:=\abs{\bL}$, it holds that $\kappa\leq 2\Delta$.  We define the tensorized Laplacian as the matrix $\bar{\bL}=\bL\otimes I_{m}$. We set $\bar{b}=(b_1,\ldots,b_{N})^{\top}$, $\boldsymbol x=\op{col}(x_{1},\ldots,x_{N})$ and similarly $\boldsymbol z$ and $\boldsymbol \lambda$.

Let ${\bf{A}}=\op{diag}(A_1,\dots,A_N)$ and define 
%\vspace{-.15cm}\begin{equation}\begin{aligned}
%\bA:=\left[\begin{array}{cccc} 
%A_{1} & 0 & \ldots & 0\\
%0 & A_{2} & \ldots & 0\\
%\vdots & \vdots & \ddots & \vdots\\
%0 & 0 & \ldots & A_{N}
%\end{array}\right]. 
%\end{aligned}\vspace{-.15cm}\end{equation}
%and the step size matrices $\rho,\sigma,\tau$ of suitable dimensions. Define
%the linear operators $\mc A,\mc B:\setH\to\setH$, by  
%\vspace{-.15cm}\begin{equation}\begin{aligned}
%\mc A(x,z,\lambda)&:=\left[F(\boldsymbol x);\0;\bar{\bL}z+\bar{b}\right],\text{ and }\\
%\mc B(x,z,\lambda)&:=\left[\bA^{\top}z,\bar{\bL}z,-\bA x-\bar{\bL}y\right],
%\end{aligned}\vspace{-.15cm}\end{equation}

\vspace{-.15cm}\begin{equation}\label{eq:AB}
\begin{aligned}
\mc{A}(\boldsymbol{x},\bs z,\bs \lambda):=\op{col}(F({\bf x}),{\bf{0}}_{mN},\bar{\bf{L}}\bs\lambda+\bar b),\\
\mc{B}(\boldsymbol{x},\bs z,\bs \lambda):=\op{col}({\bf{A}}^{\top}\bs \lambda, \bar{\bf{L}}\bs \lambda,-{\bf{A}}\boldsymbol{x}-\bar{\bf{L}}\bs z)
%\mc{A} &:\left[\begin{array}{l}
%\boldsymbol{x} \\
%\bs z\\
%\bs \lambda
%\end{array}\right] \mapsto\left[\begin{array}{c}
%F(\boldsymbol{x}) \\ 
%{\bf{0}}_{mN}\\
%\bar{\bf{L}}\bs\lambda+\bar b
%\end{array}\right],\\
%\mc{B} &:\left[\begin{array}{l}
%\boldsymbol{x} \\ 
%\bs z\\
%\bs \lambda
%\end{array}\right] \mapsto\left[\begin{array}{ccc}
%0 & 0 & {\bf{A}}^{\top} \\ 
%0 & 0 & \bar{\bf{L}}\\
%-{\bf{A}} & -\bar{\bf{L}} & 0
%\end{array}\right]\left[\begin{array}{l}
%\boldsymbol{x} \\ 
%\bs z\\
%\bs \lambda
%\end{array}\right].
\end{aligned}
\vspace{-.15cm}\end{equation}
Let us also define the operator $\mc D :=\mc A+\mc B$, and the set-valued operator
%\vspace{-.15cm}\begin{equation}\label{eq:C}
%\mc C :\left[\begin{array}{l}
%\boldsymbol{x} \\ 
%\bs z\\
%\bs \lambda
%\end{array}\right] \mapsto
%\left[\begin{array}{c}
%G(\boldsymbol{x}) \\ 
%{\bf{0}}_{mN}\\
%\NC_{\RR_{ \geq 0}^{mN}}(\mu)
%\end{array}\right].
%\vspace{-.15cm}\end{equation}
\vspace{-.15cm}\begin{equation}\label{eq:C}
\mc C(\boldsymbol{x},\bs z,\bs \lambda)=G(\boldsymbol{x})\times\{{\bf{0}}_{mN}\}\times \NC_{\RR_{ \geq 0}^{mN}}(\bs \lambda)
\vspace{-.15cm}\end{equation}
where $G(\bs x)=\partial g_{1}(x_{1})\times\cdots\times\partial g_{N}(x_{N})$.
%as
%\vspace{-.15cm}\begin{equation}
%\mc D:\left[\begin{array}{l}
%\boldsymbol{x} \\ 
%\bs z\\
%\bs \lambda
%\end{array}\right]\mapsto\left[\begin{array}{c}
%F(\boldsymbol x)+\bA^{\top}\bs \lambda\\
%-\bar{\bL}\lambda\\
%-\bA \bs x+\bar{\bL}(\bs \lambda-\bs z)+\bar{b}
%\end{array}\right].
%\vspace{-.15cm}\end{equation}

Let us summarize the properties of the operators above.
\begin{lemma}\label{lemma_op}
The following statements hold:
\begin{itemize}
\item[(i)] $\mc A$ is maximally monotone and $L_{\mc A}=( \frac{1}{\beta}+\kappa)$-Lipschitz continuous.
\item[(ii)] $\mc B$ is maximally monotone and $L_{\mc B}=(2|{\bf {A}}|+2\kappa)$-Lipschitz continuous.
%\item $\mc C$ is maximally monotone.
\item[(iii)] $\mc D$is maximally monotone and $L_{\mc D}=L_{\mc A}+L_{\mc B}$-Lipschitz continuous.
\item[(iv)] $\mc C+\mc D$ is maximally monotone.
\end{itemize}
\end{lemma}
\begin{proof}
(i) The operator $\mc A$ is maximally monotone by \cite[Prop. 20.23]{BauCom16}.
Furthermore, given $\kappa=\abs{\bL}$ and $\bs u:=\op{col}(\bs x,\bs z,\bs \lambda)$, it holds that 
\vspace{-.15cm}\begin{equation}\label{A_lip}
\begin{aligned}
\norm{ \mc A\bs u-\mc A\bs u'}\leq& \norm{F(\bs x)-F(\bs x')}+\norm{\bar{\bL}(\bs z-\bs z')}\\
\leq&( \frac{1}{\beta}+\kappa)\left(\norm{\bs x-\bs x'}+\norm{\bs z-\bs z'}\right),
\end{aligned}
\vspace{-.15cm}\end{equation}
showing that $\mc A$ is $L_{\mc A}:=(\frac{1}{\beta}+\kappa)$-Lipschitz continuous.\\
(ii) The operator $\mc B$ is maximally monotone by \cite[Cor. 20.28]{BauCom16}.
By a computation similar to \eqref{A_lip}, it can be shown that $\mc B$ is $L_{\mc B}=(2|{\bf {A}}|+2\kappa)$-Lipschitz continuous. \\
(iii) The operator $\mc D$ is maximally monotone since $\dom(\mc B)=\bs\Omega$ \cite[Prop. 20.23]{BauCom16}. It is $L_{\mc D}=L_{\mc A}+L_{\mc B}$-Lipschitz continuous because sum of Lipschitz operators.\\
(iv) The operator $\mc C$ is maximally monotone by \cite[Lem. 5]{yi2019} and $\mc C+\mc D$ is maximally monotone by \cite[Prop. 20.23]{BauCom16}.
\end{proof}

The following result, which immediately follows from the definition of the inner product $\langle\cdot,\cdot\rangle_\Phi$, holds for monotone operators and it will be recalled later on.
\begin{lemma}\label{lemma_mono}
Let $\Phi\succ0$ and $\mc T$ be a monotone operator, then $\Phi^{-1}\mc T$ is monotone in the Hilbert space $\mc H_\Phi$.
\fineass\end{lemma}
%\begin{proof}
%It follows from the definition of the inner product $\langle\cdot,\cdot\rangle_\Phi$.%:\vspace{-.15cm}
%\vspace{-.15cm}\begin{equation*}
%\langle\Phi^{-1}\mc Tx-\Phi^{-1}\mc Ty,x-y\rangle_\Phi=\langle\mc Tx-\mc Ty,x-y\rangle\geq0.\vspace{-.4cm}
%\vspace{-.15cm}\end{equation*}
%\end{proof}

%\begin{assumption}\label{ass:coercive}
%We assume that 
%\vspace{-.15cm}\begin{equation*}
%\lim_{\norm{\bs u}\to\infty}\norm{\mc C\bs u+\mc D\bs u}=\infty.
%\vspace{-.15cm}\end{equation*}
%\end{assumption}
%This coercivity assumption is quite mild. In the standard GNE model in which the idiosyncratic cost functions $g_{i}$ are indicator functions for compact convex sets $\mc X_{i}\subset\R^{n_i}$ this assumption clearly holds. 

Now, given the operators $\mc A$, $\mc B$ and $\mc C$ as in \eqref{eq:AB} and \eqref{eq:C}, a simple proof shows that the zeros of the sum $\mc A+\mc B+\mc C$ are v-GNE of the game in \eqref{game}. In fact, this follows verbatim from \cite[Thm. 2]{yi2019}, so we omit the detailed proof here. %The following notation is used: $\boldsymbol x=\op{col}(x_{1},\ldots,x_{N})$ is the action profile of all the agents, $\boldsymbol z=\op{col}(z_{1},\ldots,z_{N})$ indicates the auxiliary consensus enforcing variable and $\boldsymbol \lambda=\op{col}(\lambda_{1},\ldots,\lambda_{N})$ is the local dual variable.}
\begin{lemma}
The set $\Zer(\mc A+\mc B+\mc C)$ is the set of v-GNE of the game in \eqref{game} and it is non-empty.
\fineass
\end{lemma}

%\begin{proof}
%Let $\mc Z:=\Zer(\mc A+\mc B+\mc C)$. Existence follows from \cite[Prop. 23.36]{BauCom16}. To show that elements of $\mc Z$ are v-GNE, we proceed as follows. Let $\bs u^{\ast}=(\bs x^{\ast},\bs z^{\ast},\bs \lambda^{\ast})\in\mc Z$, i.e. $-\mc D \bs u^{\ast}\in \mc C \bs u^{\ast}$.
%Writing out this condition explicitly gives 
%\vspace{-.15cm}\begin{equation*}
%\forall i\in\mc I\quad\begin{cases}
%0\in \nabla_{x_{i}}f_{i}(\bs x^{\ast})+\partial g_{i}(x^{\ast}_{i})+A^{\top}_{i}\lambda_{i}\\
%0=\bar{\bL}\bs \lambda,\\
%0\in -(\bA \bs x^{\ast}-\bar{b})-\bar{\bL}(\bs z-\bs \lambda)+\NC_{\RR^{mN}_{\geq 0}}(\bs \lambda).
%\end{cases}
%\vspace{-.15cm}\end{equation*}
%The second condition implies that $\lambda=\1_{N}\otimes \lambda^{\ast}$ for some $\lambda^{\ast}\in\R^{m}$. Thus, $\lambda^{\ast}_{i}=\lambda^{\ast}$ for all $i\in\mc I$. Summing %the third condition over all the agents gives the complementary slackness condition $0\in b-A \bs x^{\ast}+\NC_{\R^{m}_{\geq0}}(\bs \lambda^{\ast}).$ Therefore, the pair $(\bs x^{\ast},\bs \lambda^{\ast})$ is a v-GNE.
%\end{proof}

\subsection{Forward-backward operator splitting} 
\label{sec:FB}
The aim of this section is to revisit a distributed forward-backward (FB) splitting algorithm for the  distributed computation of a v-GNE, see Algorithm \ref{FB_algo} \cite{yi2019}. 
%%%%%%%%%%%
\begin{algorithm}[h]
\caption{Preconditioned Forward Backward}\label{FB_algo}
Initialization: $x_i^0 \in \Omega_i, \lambda_i^0 \in \RR_{\geq0}^{m},$ and $z_i^0 \in \RR^{m} .$\\
Iteration $k$: Agent $i$\\
($1$) Receives $x_j^k$ for $j \in \mathcal{N}_{i}^{J}, \lambda_j^k$ for $j \in \mathcal{N}_{i}^{\lambda}$, then updates
$$\begin{aligned}
&x_i^{k+1}=\op{prox}^{\rho_i}_{g_{i}}[x_i^k-\rho_{i}(\nabla_{x_{i}} f_{i}(x_i^k,\boldsymbol x_{-i}^k)-A_{i}^{T} \lambda_i^k)]\\
&z_i^{k+1}=z_i^k+\sigma_{i} \sum\nolimits_{j \in \mathcal{N}_{i}^{\lambda}} w_{i,j}(\lambda_i^k-\lambda_j^k)
\end{aligned}$$
($2$) Receives $z_{j}^{k+1}$ for $j \in \mathcal{N}_{i}^{\lambda}$, then updates
$$\begin{aligned}
 \lambda_i^{k+1}=&\op{proj}_{\RR^m_{\geq 0}}\{\lambda_i^k-\tau_{i}[A_{i}(2 x_i^{k+1}-x_i^k)-b_{i}\\
&+\sum\nolimits_{j \in \mathcal{N}_{i}^{\lambda}} w_{i,j}[2(z_i^{k+1}-z_{j, k+1})-(z_i^k-z_j^k)]\\
&+\sum\nolimits_{j \in \mathcal{N}_{i}^{\lambda}} w_{i,j}(\lambda_i^k-\lambda_j^k)]\}\\
\end{aligned}$$
\end{algorithm}
%%%%%%%%%%%%%
From now on, the triplet $\bs u:=\op{col}(\bs x,\bs z,\bs \lambda)$ defines the state variable of a distributed algorithm. Given the state at iteration $k$, $\bs u^k=(\bs x^k,\bs z^k,\bs \lambda^k)$, the FB algorithm can be written as fixed-point iteration of the form $\bs u^{k+1}=T_{\text{FB}}\bs u^k,$ where
\vspace{-.15cm}\begin{equation}\label{eq:FB}
T_{\text{FB}}:= \mathrm{J}_{\Phi_{\text{FB}}^{-1}(\mc C+\mc B)}(\op{Id}-\Phi_{\text{FB}}^{-1}\mc A).
\vspace{-.15cm}\end{equation}
and $\Phi_{\text{FB}}$ is the preconditioning matrix defined as 
\vspace{-.15cm}\begin{equation}\label{eq:phi}
\Phi_{\text{FB}}:=\left[\begin{array}{ccc} 
\rho^{-1} & 0 & -\bA^{\top}\\
0 & \sigma^{-1} & -\bar{\bL}\\
-\bA & -\bar{\bL} &  \tau^{-1}
\end{array}\right].
\vspace{-.15cm}\end{equation}
The matrices $\rho=\diag\{\rho_{1}I_{n_{1}},\ldots,\rho_{N} I_{n_{N}}\}$, $\sigma$ and $\tau$ (defined analoguosly)
%\vspace{-.15cm}\begin{equation}\label{stepsizes}
%\begin{aligned}
%\rho&=\diag\{\rho_{1}I_{n_{1}},\ldots,\rho_{N} I_{n_{N}}\},\\
%\sigma&=\diag\{\sigma_{1}I_{m},\ldots,\sigma_{N}I_{m}\},\\
%\tau&=\diag\{\tau_{1}I_{m},\ldots,\tau_{N}I_{m}\}
%\end{aligned}
%\vspace{-.15cm}\end{equation}
collect the step sizes of the primal, the auxiliary and the dual updates, respectively. By choosing the step sizes appropriately, the preconditioning matrix $\Phi$ can be made positive definite \cite{belgioioso2018}. The FB algorithm is known to converge to a zero of a monotone inclusion $0 \in \mc A+\mc B+\mc C$ when the operators are maximally monotone and the single-valued operator $\Phi_{\text{FB}}^{-1}\mc A$ is cocoercive \cite[Thm. 26.14]{BauCom16}. Thus, the pseudo-gradient mapping $F$ in \eqref{eq:F} should satisfy the following assumption.
\begin{assumption}[Strong monotonicity]
\label{ass:Hstrong}
For all $\bs x,\bs x'\in \bs{\Omega}$,
%\vspace{-.15cm}\begin{equation*}
$\inner{F(\boldsymbol x)-F(\boldsymbol{x}'),\bs x-\bs x'}\geq\eta\norm{\boldsymbol{x}-\boldsymbol{x}'}^{2}$,%\vspace{-.3cm}
%\vspace{-.15cm}\end{equation*}
for some $\eta>0$.
\fineass
\end{assumption}

To ensure the cocoercivity condition, we refer to the following result.

\begin{lemma}\cite[Lem. 5 and Lem. 7]{yi2019}\label{lemma_coco}
Let $\Phi\succ0$ and $F$ as in \eqref{eq:F} satisfy Assumption \ref{ass:Hstrong}. Then, the following hold:
\begin{itemize}
\item[(i)] $\mc A$ is $\theta$-cocoercive with $\theta\leq\min\{1/2\Delta,\eta\beta^2\}$.
\item[(ii)] $\Phi^{-1}\mc A$ is $\alpha\theta$-cocoercive with $\alpha=1/\abs{\Phi^{-1}}$. 
\fineass
\end{itemize}
\end{lemma}

We recall that convergence to a v-GNE has been demonstrated in \cite[Th. 3]{yi2019}, if the step sizes in \eqref{eq:phi} are chosen small enough \cite[Lem. 6]{yi2019}.

\subsection{Forward-backward-forward splitting}
\label{sec:FBF}
In this section, we propose our distributed forward-backward-forward (FBF) scheme, Algorithm \ref{FBF_algo}.

%\begin{algorithm}
%\caption{Distributed Forward Backward Forward}\label{FBF_algo}
%$$\begin{aligned}
%&\tilde{x}^{k}=\prox^{\rho}_{g}\left(\boldsymbol x^{k}-\rho(H(\boldsymbol x^{k})+\bA^{\top}\boldsymbol \lambda^{k})\right)\\
%&\tilde{z}^{k}=z^k-\sigma\bar{\bL}\boldsymbol \lambda^{k},\\
%&\tilde{\lambda}^{k}=\Pi_{\RR^{mN}_{\geq 0}}\left[\boldsymbol \lambda^{k}+\tau(\bA \boldsymbol x^{k}-\bar{b})+\tau\bar{\bL}(z^k-\boldsymbol \lambda^{k})\right],\\
%&x^{k+1}=\tilde{x}^{k}+\rho\left(H(\boldsymbol x^{k})-H(\tilde{x}^{k})+\bA^{\top}(\boldsymbol \lambda^{k}-\tilde{\lambda}^{k})\right),\\
%&z^{k+1}=\tilde{z}^{k}+\sigma\bar{\bL}(\boldsymbol \lambda^{k}-\tilde{\lambda}^{k}),\\
%&\lambda^{k+1}=\tilde{\lambda}^{k}+\tau\left(\bA(\tilde{x}^{k}-\boldsymbol x^{k})-\bar{\bL}(z^k-\tilde{z}^{k})+\bar{\bL}(\boldsymbol \lambda^{k}-\tilde{\lambda}^{k})\right).
%\end{aligned}$$
%\end{algorithm}
\begin{algorithm}
\caption{Distributed Forward Backward Forward}\label{FBF_algo}
Initialization: $x_i^0 \in \Omega_i, \lambda_i^0 \in \RR_{\geq0}^{m},$ and $z_i^0 \in \RR^{m} .$\\
Iteration $k$: Agent $i$\\
($1$) Receives $x_j^k$ for $j \in \mathcal{N}_{i}^{J}$, $ \lambda_j^k$ and $z_{j,k}$ for $j \in \mathcal{N}_{i}^{\lambda}$ then updates
$$\begin{aligned}
&\tilde x_i^k=\op{prox}^{\rho_i}_{g_{i}}[x_i^k-\rho_{i}(\nabla_{x_{i}} f_{i}(x_i^k,\boldsymbol x_{-i}^k)-A_{i}^{T} \lambda_i^k)]\\
&\tilde z_i^k=z_i^k+\sigma_{i} \sum\nolimits_{j \in \mathcal{N}_{i}^{\lambda}} w_{i,j}(\lambda_i^k-\lambda_j^k)\\
&\tilde\lambda_i^k=\op{proj}_{\RR^m_{\geq 0}}\{\lambda_i^k-\tau_{i}(A_{i}x_i^k-b_{i})\\
&\quad\quad+\tau\sum\nolimits_{j \in \mathcal{N}_{i}^{\lambda}} w_{i,j}[(z_{i}^{k}-z_j^k)-(\lambda_i^k-\lambda_j^k)]\}
\end{aligned}$$
($2$) Receives $\tilde x_j^k$ for $j \in \mathcal{N}_{i}^{J}$, $ \tilde \lambda_j^k$and $\tilde z_{j}^{k}$ for $j \in \mathcal{N}_{i}^{\lambda}$ then updates
$$\begin{aligned}
&x_i^{k+1}=\tilde x_i^k-\rho_{i}(\nabla_{x_{i}} f_{i}(x_i^k,\boldsymbol x_{-i}^k)-\nabla_{\tilde x_{i}} f_{i}(\tilde x_i^k,\tilde {\boldsymbol{x}}_{-i}^k))\\
&\quad\quad\quad-\rho_iA_{i}^{T} (\lambda_i^k-\tilde \lambda_{i,k})\\
&z_i^{k+1}=\tilde z_i^k+\sigma_{i} \sum\nolimits_{j \in \mathcal{N}_{i}^{\lambda}} w_{i,j}[(\lambda_i^k-\lambda_j^k)-(\tilde\lambda_i^k-\tilde\lambda_j^k)]\\
&\lambda_i^{k+1}=\tilde{\lambda}_i^{k}+\tau_iA_i(\tilde{x}_{i}^{k}-x_{i}^{k})\\
&\quad\quad\quad-\tau_i\sum\nolimits_{j \in \mathcal{N}_{i}^{\lambda}} w_{i,j}[(z_i^k-z_j^k)-(\tilde z_i^k-\tilde z_j^k)]\\
&\quad\quad\quad+\tau_i\sum\nolimits_{j \in \mathcal{N}_{i}^{\lambda}} w_{i,j}[(\lambda_{i,k}-\lambda_j^k)-(\tilde\lambda_i^k-\tilde\lambda_j^k)]\\
\end{aligned}$$
\end{algorithm}

In compact form, the FBF algorithm generates two sequences $(\bs u^{k},\bs v^{k})_{k\geq 0}$ as follows: 
\vspace{-.15cm}\begin{equation}\label{FBF}
\begin{aligned}
\bs u^{k}&=J_{\Psi^{-1} \mc C}(\bs v^{k}-\Psi^{-1} \mc D \bs v^{k})\\
\bs v^{k+1}&=\bs u^{k}+\Psi^{-1} (\mc D\bs v^{k}-\mc D\bs u^{k}).
\end{aligned}
\vspace{-.15cm}\end{equation}

In \eqref{FBF}, $\Psi$ is the block-diagonal matrix of the step sizes: 
\vspace{-.15cm}\begin{equation}\label{Psi}
\Psi=\op{diag}(\rho^{-1},\sigma^{-1}, \tau^{-1}),
\vspace{-.15cm}\end{equation}
%\tcb{with $\rho$, $\sigma$ and $\tau$ being the step sizes diagonal matrices.}

We recall that $\mc D=\mc A+\mc B$ is single-valued, maximally monotone and Lipschitz continuous by Lemma \ref{lemma_op}. Each iteration differs from the scheme in \eqref{eq:FB} by one additional forward step and the fact that the resolvent is now defined in terms of the operator $\mc C$ only. Writing the coordinates as $\bs u^{k}=(\tilde{\boldsymbol{x}}^{k},\tilde{\bs z}^{k},\tilde{\bs \lambda}^{k})$ and $\bs v^{k}=(\boldsymbol{x}^{k},\boldsymbol z^{k},\boldsymbol \lambda^{k})$, the updates are explicitly given in Algorithm \ref{FBF_algo}.

FBF operates on the splitting $\mc C+\mc D$ and it can be compactly written as the fixed-point iteration
%\vspace{-.15cm}\begin{equation}\label{eq:FBF}
$\bs v^{k+1}=T_{\text{FBF}} \, \bs v^{k},$
%\vspace{-.15cm}\end{equation}
where the mapping $T_{\text{FBF}}$ is defined as  
\vspace{-.15cm}\begin{equation}\label{eq:T_FBF}
T_{\text{FBF}}:=\Psi^{-1}\mc D+(\Id-\Psi^{-1}\mc D)\circ \op{J}_{\Psi^{-1}\mc C}\circ(\Id-\Psi^{-1}\mc D). 
\vspace{-.15cm}\end{equation}
%The step size matrix $\Psi$, though, requires some small adaptations. 

%%Let $\rho:=\max_{i}\rho_{i},\sigma:=\max_{i}\sigma_{i},\tau:=\max_{i}\tau_{i}$. Then $\abs{\Psi^{-1}}=\max\{\rho,\sigma,\tau\}$. Since the matrix $\Psi$ is block diagonal, observe that 
%%%\vspace{-.15cm}\begin{equation*}
%%%\begin{aligned}
%%$\norm{u}^{2}_{\Psi}\geq\min\{1/\rho,1/\sigma,1/\tau\}\norm{u}^2.$
%%%\end{aligned}
%%%\vspace{-.15cm}\end{equation*}
%%We call $\alpha:=\min\{1/\rho,1/\sigma,1/\tau\}>0$.
%%Thus, $\alpha$ is the smallest eigenvalue of the matrix $\Psi$, and $\alpha=1/\abs{\Psi^{-1}}$. 

To ensure convergence of Algorithm \ref{FBF_algo} to a v-GNE of the game in \eqref{game}, we need the next assumption.

\begin{assumption}\label{step_FBF}
$\abs{\Psi^{-1}} < 1/L_{\mc D}$, with $\Psi$ as in \eqref{Psi} and $L_{\mc D}$ being the Lipschitz constant of $\mc D$ as in Lemma \ref{lemma_op}.
\fineass
\end{assumption}

%We call $\op{fix}(T_{\text{FBF}})$ the set of fixed points of $T_\Psi$ in \eqref{eq:T_FBF}.

\begin{theorem}\label{theo_FBF}
Let Assumption \ref{step_FBF} hold. The sequence $(\bs x^k,\bs \lambda^k)$ generated by Algorithm \ref{FBF_algo} converges to 
$\op{zer}(\mc A+\mc B+\mc C)$, thus the primal variable converges to a v-GNE of the game in \eqref{game}.
\end{theorem}
\begin{proof}
The fixed-point iteration with $T_{\text{FBF}}$ as in \eqref{eq:T_FBF} can be derived from \eqref{FBF} by substituting $\bs u_k$.
%Then, writing explicitly the iterations of \eqref{FBF} and solving for $\tilde{\bs x}^k,\tilde{\bs z}^k,\tilde{\bs \lambda}^k$ and $\bs x^{k+1},\bs z^{k+1},\bs\lambda^{k+1}$ we obtain Algorithm \ref{FBF_algo}. Therefore Algorithm \ref{FBF_algo} is the fixed point iteration. 
Therefore, the sequence $(\bs x^k,\bs\lambda^k)$ generated by Algorithm \ref{FBF_algo} converges to a v-GNE by \cite[Th.26.17]{BauCom16} and \cite[Th.3.4]{tseng2000} since $\Psi^{-1}\mc A$ is monotone by Lemma \ref{lemma_mono} and $\mc A+\mc B+\mc C$ is maximally monotone by Lemma \ref{lemma_op}. See Appendix \ref{sec:FBF} for details.
\end{proof}

We emphasize that Algorithm \ref{FBF_algo} does not require strong monotonicity (Assumption \ref{ass:Hstrong}) of the pseudo-gradient mapping $F$ in \eqref{eq:F}.
Moreover, we note that the FBF algorithm requires two evaluations of the individual gradients, which requires computing the operator $\mc D$ twice per iteration. At the level of the individual agents, this means that we need two communication rounds per iteration in order to exchange the necessary information. Compared with the FB algorithm, the non-strong monotonicity assumption comes at the price of increased communications at each iteration.

\subsection{Forward-backward-half forward splitting}
\label{sec:FBHF}

Should the strong monotonicity condition (Assumption \ref{ass:Hstrong}) be satisfied, an alternative to the FB is the \emph{forward-backward-half-forward} (FBHF) operator splitting, developed in \cite{briceno2018}. Thus, our second GNE seeking algorithm is a distributed FBHF, described in Algorithm \ref{FBHF_algo}.

\begin{algorithm}
\caption{Distributed Forward Backward Half Forward}\label{FBHF_algo}
Initialization: $x_i^0 \in \Omega_i, \lambda_i^0 \in \RR_{\geq0}^{m},$ and $z_i^0 \in \RR^{m} .$\\
Iteration $k$: Agent $i$\\
($1$) Receives $x_j^k$ for $j \in \mathcal{N}_{i}^{J}$, $ \lambda_j^k$ and $z_{j,k}$ for $j \in \mathcal{N}_{i}^{\lambda}$ then updates
$$\begin{aligned}
&\tilde x_i^k=\op{prox}^{\rho_i}_{g_{i}}[x_i^k-\rho_{i}(\nabla_{x_{i}} f_{i}(x_i^k,\boldsymbol x_{-i}^k)-A_{i}^{T} \lambda_i^k)]\\
&\tilde z_i^k=z_i^k+\sigma_{i} \sum\nolimits_{j \in \mathcal{N}_{i}^{\lambda}} w_{i,j}(\lambda_i^k-\lambda_j^k)\\
&\tilde\lambda_i^k=\op{proj}_{\RR^m_{\geq 0}}\{\lambda_i^k-\tau_{i}(A_{i}x_i^k-b_{i})\\
&\quad\quad+\tau\sum\nolimits_{j \in \mathcal{N}_{i}^{\lambda}} w_{i,j}[(z_{i}^{k}-z_j^k)-(\lambda_i^k-\lambda_j^k)]\}
\end{aligned}$$
($2$) Receives $ \tilde \lambda_j^k$and $\tilde z_{j,k}$ for $j \in \mathcal{N}_{i}^{\lambda}$ then updates
$$\begin{aligned}
&x_i^{k+1}=\tilde x_i^k+\rho_iA_{i}^{T} (\lambda_i^k-\tilde \lambda_{i,k})]\\
&z_i^{k+1}=\tilde z_i^k+\sigma_{i} \sum\nolimits_{j \in \mathcal{N}_{i}^{\lambda}} w_{i,j}[(\lambda_i^k-\lambda_j^k)-(\tilde\lambda_i^k-\tilde\lambda_j^k)]\\
&\lambda_i^{k+1}=\tilde{\lambda}_i^{k}+\tau_iA_i(\tilde{x}_{i}^{k}-x_{i}^{k})\\
&\quad\quad\quad-\tau_i\sum\nolimits_{j \in \mathcal{N}_{i}^{\lambda}} w_{i,j}[(z_i^k-z_j^k)-(\tilde z_i^k-\tilde z_j^k)]\\
\end{aligned}$$
\end{algorithm}
In compact form, the FBHF algorithm reads as
\vspace{-.15cm}\begin{equation}\label{FBHF}
\begin{aligned}
\bs u^{k} & = \mathrm{J}_{\Psi^{-1}\mc C}(\bs v^{k}-\Psi^{-1} (\mc A+\mc B) \bs v^{k}) \\
\bs v^{k+1} & =\bs u^{k}+\Psi^{-1}(\mc B\bs v^k- \mc B\bs u^k).
\end{aligned}
\vspace{-.15cm}\end{equation}
We note that the iterates of FBHF are similar to those of the FBF, but the second forward step requires the operator $\mc B$ only. 
%In compact form, the FBHF algorithm reads as the iteration
%\vspace{-.15cm}\begin{equation}\label{FBHF}
%\begin{aligned}
%\bs u^{k} & = \mathrm{J}_{\Psi^{-1}\mc C}(\bs v^{k}-\Psi^{-1} (\mc A+\mc B) \bs v^{k}) \\
%\bs v^{k+1} & =\bs u^{k}+\Psi^{-1}(\mc B\bs v^k- \mc B\bs u^k),
%\end{aligned}
%\vspace{-.15cm}\end{equation}
%Convergence of Algorithm \ref{FBHF_algo} follows from \cite{briceno2018} under cocoercivity. As shown for the Forward backward algorithm, we know that $\Phi^{-1}\mc A$ is cocoercive under Assumption \ref{ass:Hstrong}.
More simply, we can write the FBHF as the fixed-point iteration
%\vspace{-.15cm}\begin{equation}\label{eq:FBHF}
$\bs v^{k+1}=T_{\text{FBHF}}\bs v^{k},$
%\vspace{-.15cm}\end{equation}
where 
\vspace{-.15cm}\begin{equation}\label{eq:T_FBHF}
T_{\text{FBHF}}=(\Id-\Psi^{-1}\mc B)\circ \op{J}_{\Psi^{-1}\mc C}\circ (\Id-\Psi^{-1}\mc D)+\Psi^{-1}\mc B. 
\vspace{-.15cm}\end{equation}

Also in this case, we have a bound on the step sizes.

\begin{assumption}\label{step_FBHF}
$|\Psi^{-1}| \leq \min\{2\theta_{\mc A},1/L_{\mc B}\}$,
with $\theta_{\mc A}$ as in Lemma \ref{lemma_coco} and $L_{\mc B}$ as in Lemma \ref{lemma_op}.
\fineass
\end{assumption}

We note that in Assumption \ref{step_FBHF}, the step sizes in $\Psi$ can be chosen larger compared to those in Assumption \ref{step_FBF}, since the upper bound is related to the Lipschitz constant of the operator $\mc B$, not of $L_{\mc D}=L_{\mc A}+L_{\mc B}$ as for the FBF (Assumption \ref{step_FBF}). A similar comparison can be done with respect to the FB algorithm. Intuitively, larger step sizes should be beneficial in term of convergence speed.

We can now establish our convergence result for the FBHF algorithm.

\begin{theorem}
Let Assumptions \ref{ass:Hstrong} and \ref{step_FBHF} hold. The sequence $(\bs x^k,\bs \lambda^k)$ generated by Algorithm \ref{FBHF_algo} converges to $\op{zer}(\mc A+\mc B+\mc C)$, thus the primal variable converges to 
a v-GNE of the game in \eqref{game}. \fineass
\end{theorem}

\begin{proof}
%The fixed-point iteration in with $T_{\text{FBHF}}$ as in \eqref{eq:T_FBHF} corresponds to the scheme in \eqref{FBHF} using the definition of $\bs u^k$. Expanding the iterations in \eqref{FBHF} with $\Psi$ as in \eqref{Psi} and solving for $\tilde{\bs x}^k,\tilde{\bs z}^k,\tilde{\bs \lambda}^k$ and $\bs x^{k+1},\bs z^{k+1},\bs \lambda^{k+1}$ we obtain exactly the steps in Algorithm \ref{FBHF_algo}. Therefore 
Algorithm \ref{FBHF_algo} is the fixed point iteration in \eqref{eq:T_FBHF} whose convergence is guaranteed by \cite[Th. 2.3]{briceno2018} under Assumption \ref{step_FBHF} because $\Psi^{-1}\mc A$ is cocoercive by Lemma \ref{lemma_coco}. See Appendix \ref{sec:FBHF} for details.
\end{proof}

\section{Case study and numerical simulations}
We consider a networked Cournot game with market capacity constraints \cite{yi2019}. %, with $N$ companies that operate over a set of $m$ markets. Each company decides the quantity $x_i$ of product to deliver in the $n_i$ markets it is connected with. Each company has a local cost function $c_i(x_i)$ related to the production process. Each market has a bounded capacity $b_j$ so that the collective constraints are given by $A{\boldsymbol{x}}\leq b$, where $A=[A_1,\dots,A_N]$ and $A_i$ specifies in which market company $i$ participates. Each market has a price, collected in the mapping $P:\RR^m\to\RR^m$. %In general, $P$ is supposed to be a linear function. The cost function of each agent reads as $f_i(x_i,\boldsymbol x_{-i})=c_i(x_i)-P^\top(A\boldsymbol{x})A_ix_i.$
%If $c_i(x_i)$ is strongly convex with Lipschitz continuous gradient and the prices are linear, the pseudo gradient of $f_i$ is strongly monotone.
As a numerical setting, we use a set of 20 companies and 7 markets, similarly to \cite{yi2019}. Each company $i$ has a local constraint $x_i\in(0,\delta_i)$ where each component of $\delta_i$ is randomly drawn from $[1, 1.5]$. The maximal capacity of each market $j$ is $b_j$, randomly drawn from $[0.5, 1]$. The local cost function of company $i$ is $c_i(x_i) = \pi_i\sum_{j=1}^{n_i} ([x_i]_j)^2 + r^\top x_i$, where $[x_i]_j$ indicates the $j$ component of $x_i$.
For all $i\in\mc I$, $\pi_i$ is randomly drawn from $[1, 8]$, and the components of $r_i$ are randomly drawn from $[0.1, 0.6]$. Notice that $c_i(x_i)$ is strongly convex with Lipschitz continuous gradient. The price is taken as a linear function $P= \bar P-DA\boldsymbol x$ where each component of $\bar P =\op{col}(\bar P_1,\dots,\bar P_7)$ is randomly drawn from $[2,4]$ while the entries of $D=\op{diag}(d_1,\dots,d_7)$ are randomly drawn from $[0.5,1]$. Recall that the cost function of company $i$ is influenced by the variables of the agents selling in the same market. Such informations can be retrieved from \cite[Fig. 1]{yi2019}. Since $c_i(x_i)$ is strongly convex with Lipschitz continuous gradient and the prices are linear, the pseudo gradient of $f_i$ is strongly monotone. The communication graph $\mc G^\lambda$ for the dual variables is a cycle graph with the addition of the edges $(2,15)$ and $(6,13)$. As local cost functions $g_i$ we use the indicator functions. In this way, the proximal step is a projection on the local constraints sets.

The aim of these simulations is to compare the proposed schemes.
The step sizes are taken differently for every algorithm. In particular, we take $\rho_{\text{FB}}$, $\sigma_{\text{FB}}$ and $\tau_{\text{FB}}$ as in \cite[Lem. 6]{yi2019}, $\rho_{\text{FBF}}$, $\sigma_{\text{FBF}}$ and $\tau_{\text{FBF}}$ such that Assumption \ref{step_FBF} is satisfied and $\rho_{\text{FBHF}}$, $\sigma_{\text{FBHF}}$ and $\tau_{\text{FBHF}}$ such that Assumption \ref{step_FBHF} holds. We select them to be the maximum possible.

The initial points $\lambda_i^0$ and $z_i^0$ are set to 0 while the local decision variable $x_i^0$ is randomly taken in the feasible sets.

The plots in Fig. \ref{distance_sol} show the performance parameter $\frac{\norm{\boldsymbol{x}_{k+1}-\boldsymbol{x}^{*}}}{\norm{\boldsymbol{x}^{*}}}$, that is, the convergence to a solution $\bs x^*$, and the CPU time (in seconds) used by each algorithm. We run 10 simulations, changing the parameters of the cost function to show that the result are replicable. The darker line represent the average path towards the solution.

The plot in Fig \ref{distance_sol} shows that with suitable parameters convergence to a solution is faster with the FBF algorithm which, however, is computationally more expansive than the FB and FBHF algorithms.%. On the other hand Algorithm \ref{FBF_algo} is computationally more expensive than Algorithm \ref{FBHF_algo} and Algorithm \ref{FB_algo}. 

\begin{figure}[h]
\centering
\includegraphics[scale=.22]{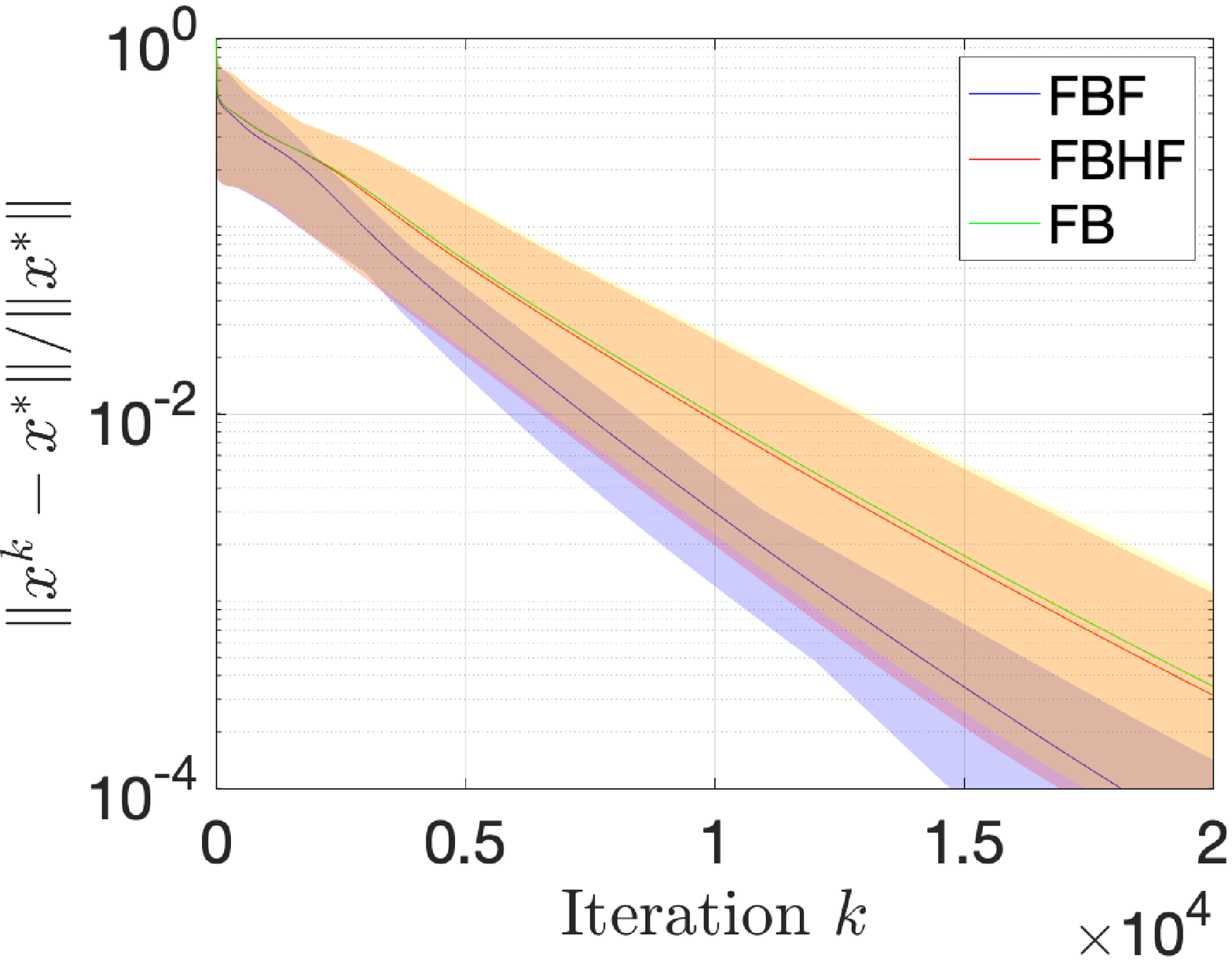}\hspace{-.2cm}
\includegraphics[scale=.22]{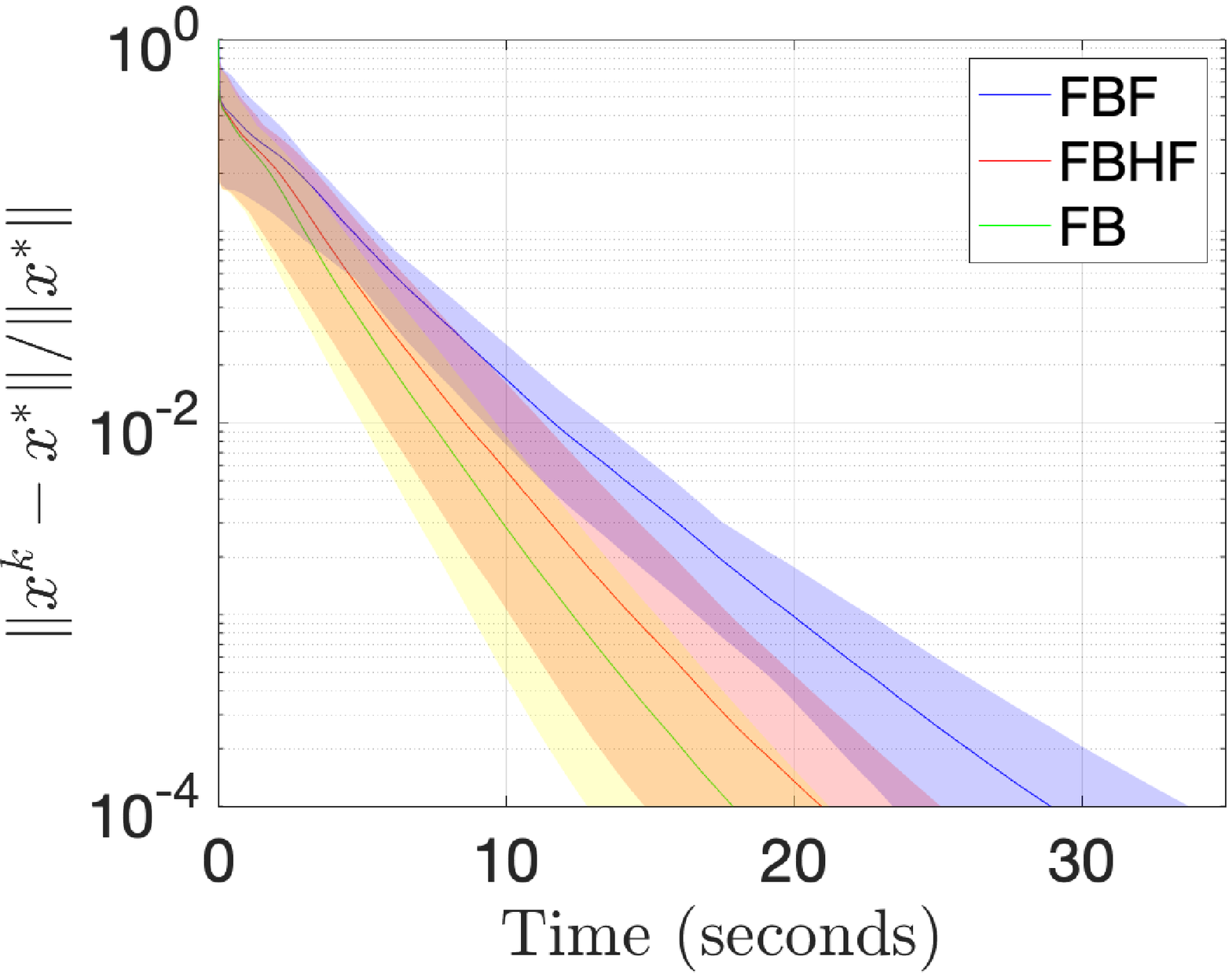}
\caption{Relative distance from v-GNE (left) and cumulative CPU time (right).}\label{distance_sol}
\end{figure}
\vspace{-.2cm}

\section{Conclusion}

The FBF and the FBHF splitting methods generate distributed equilibrium seeking algorithms for solving generalized Nash equilibrium problems. Compared to the FB, the FBF has the advantage to converge under the non-strong monotonicity assumption. This comes at the price of increased communications between the agents. If strong monotonicity holds, an alternative to the FBF is the FBHF that, in our numerical experience is less computationally expensive than the FBF.

\section{Appendix}
\subsection{Convergence of the forward-backward-forward}\label{app:FBF}
We show the convergence proof for the FBF. From now on, $\setH=\RR^n\times\RR^{mN}\times\RR^{mN}$ and $\op{fix}(T)=\{x\in\setH:Tx=x\}$.
\begin{proposition}
If Assumption \ref{step_FBF} holds, $\op{fix}(T_{\text{FBF}})=\mc Z$. 
\end{proposition}
\begin{proof}
We first show that $\mc Z\subseteq \op{fix}(T_{\text{FBF}})$. Let $u^{\ast}\in\mc Z$: 
\vspace{-.15cm}\begin{equation*}
\begin{aligned}
0\in \mc Cu^{\ast} +\mc Du^{\ast} & \Leftrightarrow -\mc Du^{\ast} \in \mc Cu^{\ast} \\
&\Leftrightarrow u^{\ast}=J_{\Psi^{-1}\mc C}(u^{\ast}-\Psi^{-1}\mc Du^{\ast})\\
&\Leftrightarrow \Psi^{-1}\mc Du^{\ast}=\Psi^{-1}\mc DJ_{\Psi^{-1}\mc C}(u^{\ast}-\Psi^{-1}\mc Du^{\ast})\\
&\Leftrightarrow u^{\ast}=T_{\text{FBF}}u^{\ast}. 
\end{aligned}
\vspace{-.15cm}\end{equation*}
Conversely, let $u^{\ast}\in\op{fix}(T_{\text{FBF}})$. Then 
%\vspace{-.15cm}\begin{equation*}
%\begin{aligned}
$u^{\ast}-J_{\Psi^{-1}\mc C}(u^{\ast}-\Psi^{-1}\mc Du^{\ast})=
\Psi^{-1}\mc Du^{\ast}-\Psi^{-1}\mc DJ_{\Psi^{-1}\mc C}(u^{\ast}-\Psi^{-1}\mc Du^{\ast})$
%\end{aligned}
%\vspace{-.15cm}\end{equation*}
ans
\vspace{-.15cm}\begin{equation*}\begin{aligned}
&\norm{u^{\ast}-J_{\Psi^{-1}\mc C}(u^{\ast}-\Psi^{-1}\mc Du^{\ast})}\leq\\
&\leq \alpha^{-1}\norm{\mc Du^{\ast} -\mc DJ_{\Psi^{-1}\mc C}(u^{\ast}-\Psi^{-1}\mc Du^{\ast})}\\
&\leq \tfrac{L}{\alpha}\norm{u^{\ast}-J_{\Psi^{-1}\mc C}(u^{\ast}-\Psi^{-1}\mc Du^{\ast})}.
\end{aligned}\vspace{-.15cm}\end{equation*}
Hence, %if $\tfrac{L}{\alpha}<1$, we see that 
$u^{\ast}=J_{\Psi^{-1}\mc C}(u^{\ast}-\Psi^{-1}\mc Du^{\ast})$. 
\end{proof}

\begin{proposition}
For all $u^{\ast}\in\op{fix}(T_{\text{FBF}})$ and $v\in\setH$, there exists $\varepsilon\geq 0$ such that 
\vspace{-.15cm}\begin{equation}\label{eq:Fejer}
\norm{T_{\text{FBF}}v-u^{\ast}}^{2}_{\Psi}= \norm{v-u^{\ast}}^{2}_{\Psi}-\left(1-(L/\alpha)^{2}\right)\norm{u-v}^{2}_{\Psi}-2\varepsilon.
\vspace{-.15cm}\end{equation}
\end{proposition}
\begin{proof}
Let $u^{\ast}\in\op{fix}(T_{\text{FBF}})$ and $u=J_{\Psi^{-1}\mc C}(v-\Psi^{-1}\mc D v),v^{+}=T_{\text{FBF}}v$, for $v\in\setH$ arbitrary. Then, 
\vspace{-.18cm}\begin{equation*}\begin{aligned}
\norm{v-u^{\ast}}_{\Psi}^{2}&=\norm{v-u+u-v^{+}+v^{+}-u^{\ast}}^{2}_{\Psi}\\
&=\norm{v-u}^{2}_{\Psi}+\norm{u-v^{+}}^{2}_{\Psi}+\norm{v^{+}-u^{\ast}}^{2}_{\Psi}\\
&+2\inner{v-u,u-u^{\ast}}_{\Psi}+2\inner{u-v^{+},v^{+}-u^{\ast}}_{\Psi}.
\end{aligned}\vspace{-.18cm}\end{equation*}
Since, 
%\vspace{-.15cm}\begin{equation*}\begin{aligned}
$2\inner{u-v^{+},v^{+}-u^{\ast}}_{\Psi}=2\inner{u-v^{+},v^{+}-u}_{\Psi}
+2\inner{u-v^{+},u-u^{\ast}}_{\Psi}=-2\norm{u-v^{+}}_{\Psi}^{2}+2\inner{u-v^{+},u-u^{\ast}}_{\Psi}.$
%\end{aligned}\vspace{-.15cm}\end{equation*}
This gives 
%\vspace{-.15cm}\begin{equation*}\begin{aligned}
$\norm{v-u^{\ast}}_{\Psi}^{2}=\norm{v-u}^{2}_{\Psi}-\norm{u-v^{+}}_{\Psi}^{2}+\norm{v^{+}-u^{\ast}}_{\Psi}^{2}+2\inner{u-u^{\ast},v-v^{+}}_{\Psi}. $
%\end{aligned}\vspace{-.15cm}\end{equation*}
By definition of the updates, we have for $\bar{v}\equiv Bv,\bar{u}\equiv Bu,\hat{v}\in Cu$, the identities
%\vspace{-.15cm}\begin{equation*}\begin{aligned}
$u+\Psi^{-1}\hat{v}=v-\Psi^{-1}\bar{v}$ and $v^{+}=u+\Psi^{-1}(\bar{v}-\bar{u}).$
%\end{aligned}\vspace{-.15cm}\end{equation*}
Furthermore, since $0\in \mc Du^{\ast} +\mc Cu^{\ast} $, there exists $\hat{v}^{\ast}\in \mc Cu^{\ast} $ and $\bar{u}^{\ast}\equiv \mc Du^{\ast} $ such that
%\vspace{-.15cm}\begin{equation*}
$0=\bar{u}^{\ast}+\hat{v}^{\ast}.$
%\vspace{-.15cm}\end{equation*}
It follows that 
%\vspace{-.15cm}\begin{equation*}
$v-v^{+}=v-u-\Psi^{-1}(\bar{v}-\bar{u})=\Psi^{-1}(\hat{v}+\bar{u}).$
%\vspace{-.15cm}\end{equation*}
Hence, 
\vspace{-.18cm}\begin{equation*}
\begin{aligned}
\norm{v-u^{\ast}}_{\Psi}^{2}=&\norm{v-u}^{2}_{\Psi}-\norm{u-v^{+}}_{\Psi}^{2}+\\
&+\norm{v^{+}-u^{\ast}}_{\Psi}^{2}+2\inner{u-u^{\ast},\hat{v}+\bar{u}}\\
=&\norm{v-u}^{2}_{\Psi}-\norm{u-v^{+}}_{\Psi}^{2}+\norm{v^{+}-u^{\ast}}_{\Psi}^{2}+\\
&+2\inner{\hat{v}-\hat{v}^{\ast}-\bar{u}^{\ast}+\bar{u},u-u^{\ast},u-u^{\ast}}.
\end{aligned}
\vspace{-.18cm}\end{equation*}
Since $(u,\hat{v}),(u^{\ast},\hat{v}^{\ast})\in\gr(C),(u^{\ast},\bar{u}^{\ast}),(u,\bar{u})\in\gr(B)$, it follows from the monotonicity that 
%\vspace{-.15cm}\begin{equation*}
$\varepsilon:=\inner{\hat{v}-\hat{v}^{\ast}-\bar{u}^{\ast}+\bar{u},u-u^{\ast},u-u^{\ast}}\geq 0.$
%\vspace{-.15cm}\end{equation*}
Finally, observe that $u-v^{+}=\Psi^{-1}(\mc Du-\mc Dv)$, and that 
\vspace{-.15cm}\begin{equation*}\begin{aligned}
&\norm{\Psi^{-1}(\mc Du-\mc Dv)}_{\Psi}^{2}=\inner{\Psi^{-1}(\mc Du-\mc Dv),\mc Du-\mc Dv}\\
&\leq \lambda_{\max}(\Psi^{-1})\norm{\mc Du-\mc Dv}^{2}\leq L^{2}\lambda_{\max}(\Psi^{-1})\norm{u-v}^{2}\\
&\leq L^{2}\tfrac{\lambda_{\max}(\Psi^{-1})}{\lambda_{\min}(\Psi)}\norm{u-v}^{2}_{\Psi}.
\end{aligned}\vspace{-.15cm}\end{equation*}
Since $\alpha=1/\lambda_{\max}(\Psi^{-1})=\lambda_{\min}(\Psi)$, it follows from the Lipschitz continuity of the operator $B$ that
%\[
$\norm{u-v^{+}}_{\Psi}^{2}\leq (L/\alpha)^{2}\norm{u-v}^{2}_{\Psi} $
%\]
and the statement is proven.
%We conclude, 
%\vspace{-.15cm}\begin{equation*}
%\norm{T_{\text{FBF}}v-u^{\ast}}^{2}_{\Psi}= \norm{v-u^{\ast}}^{2}_{\Psi}-\left(1-(L/\alpha)^{2}\right)\norm{u-v}^{2}_{\Psi}-2\varepsilon.
%\vspace{-.15cm}\end{equation*}
\end{proof}
\begin{corollary}
If $L/\alpha<1$, the map $T_{\text{FBF}}:\setH\to\setH$ is quasinonexpansive in the Hilbert space $(\setH,\inner{\cdot,\cdot}_{\Psi})$, i.e. 
\vspace{-.15cm}\begin{equation*}
\forall v\in\setH\; \forall u^{\ast}\in\op{fix}(T_{\text{FBF}}) \;\norm{T_{\text{FBF}}v-u^{\ast}}_{\Psi}\leq\norm{v-u^{\ast}}_{\Psi}.
\vspace{-.15cm}\end{equation*}
\end{corollary}
\begin{proposition}
If Assumption \ref{step_FBF} holds, the sequence generated by the FBF algorithm, $(v^{k})_{k\geq 0}$, is bounded in norm, and all its accumulation points are elements in $\mc Z$.
\end{proposition}
\begin{proof}
Form \eqref{eq:Fejer} we deduce that $(v^{k})_{k\geq 0}$ is Fej\'{e}r monotone with respect to $\op{fix}(T_{\text{FBF}})=\mc Z$. Therefore, it is bounded norm. It remains to show that all accumulation points are in $\mc Z$. By an obvious abuse of notation, let $(v^{k})_{k\geq 0}$ denote a converging subsequence with limit $u^{\ast}$. From \eqref{eq:Fejer} it follows $\norm{u^{k}-v^{k}}_{\Psi}\to 0$, and hence $\norm{u^{k}-v^{k}}\to 0$ as $k\to\infty$. By continuity, it therefore follows as well $\norm{\mc Du^k-\mc Dv^k}\to 0$ as $k\to\infty$. Since $u^{k}=J_{\Psi^{-1}\mc C}(v^{k}-\Psi^{-1}\mc Dv^{k})$, it follows that $w^{k}:=\Psi(v^{k}-u^{k})+\mc Du^k-\mc Dv^k\in \mc Du^k+\mc Cu^{k}.$
Since $w^{k}\to 0$ and the operator $\mc C+\mc D$ is maximally monotone by Lemma \ref{lemma_op} and has a closed graph \cite[Lem. 3.2]{tseng2000}, we conclude $0\in \mc Du^{\ast} +\mc Cu^{\ast} $. Hence, $u^{\ast}\in\mc Z$.
\end{proof}

\subsection{Convergence of the forward-backward-half-forward}
\label{sec:FBHF}
We here provide the convergence proof for the FBHF.
\begin{proposition}
If Assumption \ref{step_FBHF} holds, the sequence generated by the FBHF algorithm converges to $\mc Z$.
\end{proposition}
Since, $w-u\in\Psi^{-1}\mc C u$, it follows that $(u,w-u)\in\gr(\Psi^{-1}\mc C)$. Additionally, $0\in \mc Du^{\ast} +\mc Cu^{\ast} $, implying that $(u^{\ast},-\Psi^{-1}\mc Du^{\ast})\in \gr(\Psi^{-1}\mc C)$. Monotonicity of the involved operators, implies that 
%\vspace{-.15cm}\begin{equation*}\begin{aligned}
$\inner{u-u^{\ast},w-u-\Psi^{-1}\mc Du^{\ast}}_{\Psi}\leq 0,\text{ and }
\inner{u-u^{\ast},\Psi^{-1}(\mc Bu^{\ast} -\mc Bu)}_{\Psi}\leq 0, $
%\end{aligned}\vspace{-.15cm}\end{equation*}
Using these two inequalities, we see 
\vspace{-.15cm}\begin{equation*}\begin{aligned}
&\inner{u-u^{\ast},u-w-\Psi^{-1}\mc Bu}_{\Psi}=\inner{u-u^{\ast},\Psi^{-1}\mc Au^{\ast}}_{\Psi}\\
&+\inner{u-u^{\ast},u-w-\Psi^{-1}\mc Du^{\ast}}_{\Psi}\\
&+\inner{u-u^{\ast},\Psi^{-1}(\mc Du^{\ast} -\mc Bu)}_{\Psi}\leq \inner{u-u^{\ast},\Psi^{-1}\mc Au^{\ast}}_{\Psi}
\end{aligned}\vspace{-.15cm}\end{equation*}
Therefore, 
\vspace{-.15cm}\begin{equation}\label{step}
\begin{aligned}
&2\inner{u-u^{\ast},\Psi^{-1}(\mc Bv-\mc Bu)}_{\Psi}=\\
&2\inner{u-u^{\ast},\Psi^{-1}\mc Bv+w-u}_{\Psi}+2\inner{u-u^{\ast},u-w-\Psi^{-1}\mc Bu}_{\Psi}\\
&\leq 2\inner{u-u^{\ast},\Psi^{-1}\mc Bv+w-u}_{\Psi}+2\inner{u-u^{\ast},\Psi^{-1}\mc Au^{\ast}}_{\Psi}\\
&=2\inner{u-u^{\ast},\Psi^{-1}\mc Dv+w-u}_{\Psi}\\
&+2\inner{u-u^{\ast},\Psi^{-1}(\mc Au^{\ast}-\mc Av)}_{\Psi}\\
&=2\inner{u-u^{\ast},v-u}_{\Psi}+2\inner{u-u^{\ast},\Psi^{-1}(\mc Au^{\ast}-\mc Av)}_{\Psi},
\end{aligned}
\vspace{-.15cm}\end{equation}
where in the last equality we have used the identity $w=v-\Psi^{-1}\mc Dv$. Using the cosine formula, 
%\vspace{-.15cm}\begin{equation*}
%2\inner{u^{\ast}-u,v-u}_{\Psi}=\norm{u-u^{\ast}}^{2}_{\Psi}+\norm{v-u}^{2}_{\Psi}-\norm{v-u^{\ast}}_{\Psi}^{2},
%\vspace{-.15cm}\end{equation*}
\eqref{step} becomes 
\vspace{-.15cm}\begin{equation}
\begin{aligned}\label{eq:step}
&2\inner{u-u^{\ast},\Psi^{-1}(\mc Bv-\mc Bu)}_{\Psi}\leq \norm{v-u^{\ast}}_{\Psi}^{2}-\norm{u-u^{\ast}}_{\Psi}^{2}\\
&-\norm{v-u}_{\Psi}^{2}+2\inner{u-u^{\ast},\Psi^{-1}(\mc Au^{\ast}-\mc Av)}_{\Psi}.
\end{aligned}
\vspace{-.15cm}\end{equation}
The cocoercivity of $\Psi^{-1}\mc A$ in $(\setH,\inner{\cdot,\cdot}_{\Psi})$ gives for all $\varepsilon>0$
\vspace{-.15cm}\begin{equation*}\begin{aligned}
&2\inner{u-u^{\ast},\Psi^{-1}(\mc Au^{\ast}-\mc Av)}_{\Psi}=\\
&2\inner{v-u^{\ast},\Psi^{-1}(\mc Au^{\ast}-\mc Av)}_{\Psi}+2\inner{u-v,\Psi^{-1}(\mc Au^{\ast}-\mc Av)}_{\Psi}\\
&\leq -2\alpha\theta\norm{\Psi^{-1}(\mc Au^{\ast}-\mc Av)}_{\Psi}^{2}+2\inner{u-v,\Psi^{-1}(\mc Au^{\ast}-\mc Av)}_{\Psi}\\
&=-2\alpha\theta\norm{\Psi^{-1}(\mc Au^{\ast}-\mc Av)}_{\Psi}^{2}+\tfrac{1}{\varepsilon}\norm{\Psi^{-1}(\mc Av-\mc Au^{\ast})}_{\Psi}^{2}\\
&+\varepsilon\norm{v-u}^{2}_{\Psi}-\varepsilon\norm{v-u-\tfrac{1}{\varepsilon}\Psi^{-1}(\mc Av-\mc Au^{\ast})}^{2}_{\Psi}\\
&=\varepsilon\norm{v-u}^{2}_{\Psi}-\left(2\alpha\theta-\tfrac{1}{\varepsilon}\right)\norm{\Psi^{-1}(\mc Av-\mc Au^{\ast})}_{\Psi}^{2}\\
&-\varepsilon\norm{v-u-\tfrac{1}{\varepsilon}\Psi^{-1}(\mc Av-\mc Au^{\ast})}^{2}_{\Psi}.
\end{aligned}\vspace{-.15cm}\end{equation*}
Combining this estimate with (\ref{eq:step}), we see 
\vspace{-.15cm}\begin{equation*}\begin{aligned}
&2\inner{u-u^{\ast},\Psi^{-1}(\mc Bv-\mc Bu)}_{\Psi}\leq \norm{v-u^{\ast}}_{\Psi}^{2}-\norm{u-u^{\ast}}_{\Psi}^{2}\\
&-\norm{v-u}_{\Psi}^{2}+\varepsilon\norm{v-u}^{2}_{\Psi}-\left(2\alpha\theta-\tfrac{1}{\varepsilon}\right)\norm{\Psi^{-1}(\mc Av-\mc Au^{\ast})}_{\Psi}^{2}\\
&-\varepsilon\norm{v-u-\tfrac{1}{\varepsilon}\Psi^{-1}(\mc Av-\mc Au^{\ast})}^{2}_{\Psi}.
\end{aligned}\vspace{-.15cm}\end{equation*}
Therefore, 
\vspace{-.15cm}\begin{equation*}\begin{aligned}
&\norm{v^{+}-u^{\ast}}^{2}_{\Psi}%=\norm{T_{\text{FBHF}}v-u^{\ast}}_{\Psi}^{2}\\
=\norm{u+\Psi^{-1}(\mc Bv-\mc Bu)-u^{\ast}}^{2}_{\Psi}\\
&=\norm{u-u^{\ast}}^{2}_{\Psi}+2\inner{u-u^{\ast},\Psi^{-1}(\mc Bv-\mc Bu)}_{\Psi}\\
&+\norm{\Psi^{-1}(\mc Bv-\mc Bu)}_{\Psi}^{2}\\
&\leq \norm{u-u^{\ast}}^{2}_{\Psi}+\norm{\Psi^{-1}(\mc Bv-\mc Bu)}_{\Psi}^{2}-\norm{u-u^{\ast}}_{\Psi}^{2}\\
&-\hspace{-.1cm}\left(2\alpha\theta-\tfrac{1}{\varepsilon}\right)\norm{\Psi^{-1}(\mc Av-\mc Au^{\ast})}_{\Psi}^{2}+\norm{v-u^{\ast}}_{\Psi}^{2}\hspace{-.1cm}-\hspace{-.1cm}\norm{v-u}_{\Psi}^{2}\\
&+\varepsilon\norm{v-u}^{2}_{\Psi}-\varepsilon\norm{v-u-\tfrac{1}{\varepsilon}\Psi^{-1}(\mc Av-\mc Au^{\ast})}^{2}_{\Psi}.
\end{aligned}\vspace{-.15cm}\end{equation*}
Since, 
%\vspace{-.15cm}\begin{equation*}
$\norm{\Psi^{-1}(\mc Bv-\mc Bu)}^{2}_{\Psi}\leq (L/\alpha)^{2}\norm{v-u}^{2}_{\Psi},$
%\vspace{-.15cm}\end{equation*}
the above reads as
\vspace{-.15cm}\begin{equation*}
\begin{aligned}
\norm{T_{\text{FBHF}}v-u^{\ast}}_{\Psi}^{2}\leq& \norm{v-u^{\ast}}^{2}_{\Psi}-L^{2}\left(\tfrac{1-\varepsilon}{L^{2}}-\tfrac{1}{\alpha^{2}}\right)\norm{v-u}^{2}_{\Psi}\\
&-\tfrac{1}{\alpha\varepsilon}\left(2\theta\varepsilon-\tfrac{1}{\alpha}\right)\norm{\Psi^{-1}(\mc Av-\mc Au^{\ast})}_{\Psi}^{2}\\
&-\varepsilon\norm{v-u-\tfrac{1}{\varepsilon}\Psi^{-1}(\mc Av-\mc Au^{\ast})}^{2}_{\Psi}.
\end{aligned}
\vspace{-.15cm}\end{equation*}
In order to choose the largest interval for $1/\alpha$ ensuring that the second and third terms are negative, we set %$\sqrt{1-\varepsilon}/L=2\theta\varepsilon$, delivering 
%\vspace{-.15cm}\begin{equation*}\label{eq:eps}
%$\varepsilon=\tfrac{-1+\sqrt{1+16\theta^{2}L^{2}}}{8\theta^{2}L^{2}}$.
%\vspace{-.15cm}\end{equation*}
%Calling
%\vspace{-.15cm}\begin{equation*}\label{eq:theta}
%$\chi=\sqrt{1-\varepsilon}/L=\tfrac{4\theta}{1+\sqrt{1+16\theta^{2}L^{2}}}\leq\min\{2\theta,1/L\},$
%\vspace{-.15cm}\end{equation*} 
$\chi\leq\min\{2\theta,1/L\}$.
Then,
\vspace{-.15cm}\begin{equation*}\begin{aligned}
\norm{T_{\text{FBHF}}v-u^{\ast}}_{\Psi}^{2}\leq& \norm{v-u^{\ast}}^{2}_{\Psi}-L^{2}\left(\chi^{2}-\tfrac{1}{\alpha^{2}}\right)\norm{v-u}^{2}_{\Psi}\\
&-\tfrac{2\theta}{\alpha\chi}\left(\chi-\tfrac{1}{\alpha}\right)\norm{\Psi^{-1}(\mc Av-\mc Au^{\ast})}_{\Psi}^{2}\\
&-\tfrac{\chi}{2\theta}\norm{v-u-\tfrac{2\theta}{\chi}(\Psi^{-1}(\mc Av-\mc Au^{\ast}))}^{2}_{\Psi}.
\end{aligned}\vspace{-.15cm}\end{equation*}
From here, we obtain convergence of the sequence $(v^{k})_{k\geq 0}$ as a consequence of \cite[Thm. 2.3]{briceno2018} for $1/\alpha\in(0,\chi)$.

\bibliographystyle{IEEEtran}
\bibliography{IEEEabrv,Biblio,mybib}

\end{document}